\documentclass{aic}

\usepackage{amsmath,amssymb,amsthm}
\usepackage{mathtools}
\usepackage[normalem]{ulem}
\usepackage{enumerate}
\usepackage{bbding}
\usepackage{bm}
\usepackage{hyperref}

\aicAUTHORdetails{%
  title = {The Number of Partial Steiner Systems and $d$-Partitions}, 
  author = {Remco van der Hofstad, Rudi Pendavingh, and Jorn van der Pol},
  plaintextauthor = {Remco van der Hofstad, Rudi Pendavingh, Jorn van der Pol},
  plaintexttitle = {The Number of Partial Steiner Systems and d-partitions}, 
  keywords = {Paving matroid, \texorpdfstring{$d$}{d}-partition, design, Steiner system, entropy},
}   

\aicEDITORdetails{%
   year={2022},
   number={2},
   received={1 February 2021},   
   published={21 February 2022},  
   doi={10.19086/aic.32563},      
}   

\newtheorem{lemma}{Lemma}[section]
\newtheorem{theorem}[lemma]{Theorem}
\newtheorem{proposition}[lemma]{Proposition}
\newtheorem{conjecture}[lemma]{Conjecture}
\newtheorem*{note}{Note added in proof}

\newcommand{\e}{\mathrm{e}}

\newcommand{\RR}{\mathbb{R}}
\newcommand{\ZZ}{\mathbb{Z}}

\renewcommand{\SS}{\mathbb{S}}
\newcommand{\PP}{\mathbb{P}}

\newcommand{\HH}{\mathcal{H}}

\newcommand{\dd}[1]{\,\mathrm{d}#1}

\newcommand{\Cspace}{C^1_\blacktriangle([0,1])}
\newcommand{\good}{\mathcal{G}}
\newcommand{\logapprox}{\asymp}
\newcommand{\VV}{\mathcal{V}}
\newcommand{\Xhist}[2][\lambda]{\rv{X}_{#2}^{#1}}
\newcommand{\ztail}{\mathfrak{z}}

\newcommand{\expect}[2][]{\mathsf{E}_{#1}\ifx&#2&\else\left[#2\right]\fi}
\newcommand{\condexpect}[3][0]{\expect[#1]{#2\middle| #3}}
\newcommand{\entropy}[1]{\mathcal{H}\left(#1\right)}
\newcommand{\condentropy}[2]{\entropy{#1\mid#2}}
\newcommand{\indi}[1]{\bm{1}_{#1}}
\newcommand{\pprob}{\mathsf{P}}
\newcommand{\prob}[2][]{{\mathsf{P}_{#1}\!\ifx&#2&\else\left(#2\right)\fi}}
\newcommand{\condprob}[3][]{\prob[#1]{#2\middle| #3}}
\newcommand{\rv}[1]{\bm{#1}}

\begin{document}

%
%
%
%

\begin{frontmatter}[classification=text]


\author[rvdh]{Remco van der Hofstad\thanks{The work of RvdH is supported by the Netherlands Organisation for Scientific Research (NWO) through VICI grant 639.033.806 and the Gravitation {\sc Networks} grant 024.002.003.}}
\author[rp]{Rudi Pendavingh\thanks{The work of RP is partially supported by the Netherlands Organisation for Scientific Research (NWO) through grant 603.001.211.}}
\author[jvdp]{Jorn van der Pol\thanks{The work of JvdP is partially supported by the Netherlands Organisation for Scientific Research (NWO) through grant 603.001.211. Part of the results presented here were obtained while JvdP was at Eindhoven University of Technology and appeared in his PhD thesis~\cite{VanderPol2017}.}}

\begin{abstract}
We prove asymptotic upper bounds on the number of $d$-partitions (paving matroids of fixed rank) and partial Steiner systems (sparse paving matroids of fixed rank), using a mixture of entropy counting, sparse encoding, and the probabilistic method.
\end{abstract}
\end{frontmatter}

\section{Introduction}

\subsection{Steiner systems, partitions, and matroids}

A Steiner system with parameters~$(n,d,t)$ is a set system $(E,\mathcal{H})$ consisting of a ground set~$E$ with~$n$ elements and a collection~$\mathcal{H}$ of its $t$-subsets (called its blocks) with the property that every $d$-subset is contained in a unique block. A well-known example is that of the Steiner triple system $\textrm{STS}(7)$, with parameters $(7, 2, 3)$, whose ground set and blocks correspond to the seven points and seven lines in a projective plane of order~2.

Steiner systems do not exist for all parameters; if one does, its parameters must satisfy certain natural divisibility conditions. Keevash~\cite{Keevash2014} showed that these necessary conditions are sufficient as well, provided~$n$ is sufficiently large. Moreover, he obtained estimates on the number of designs with given parameters (asymptotically as $n\to\infty$, see~\cite{Keevash2015} or Theorem~\ref{thm:keevash} below).

Here, we obtain similar estimates for several natural relaxations of the notion of Steiner system.

A partial Steiner system is obtained by relaxing the condition that each $d$-subset be contained in a unique block to the condition that each $d$-subset be contained in \emph{at most} one block. Existence of partial designs is trivial; for example, $(E, \emptyset)$ is a partial design without any blocks.
When we further relax the condition that each block have the same cardinality, we arrive at the definition of a $d$-partition. A $d$-partition in the sense of Hartmanis~\cite{Hartmanis1959} is a finite set system $(E,\mathcal{H})$ with the following properties: each $H \in \mathcal{H}$ contains at least $d$ elements, and each $d$-element subset of $E$ is contained in a unique $H \in \mathcal{H}$. In order to avoid trivialities, we shall tacitly assume that $|\mathcal{H}| \ge 2$. Each partial Steiner system $S$ with parameters $(n,d,t)$ gives rise to a $d$-partition whose sets are the blocks of $S$ together with those $d$-subsets of $E$ that are not contained in any block of $S$.
Note that $1$-partitions are (nontrivial) set partitions in the ordinary sense, while $2$-partitions are linear spaces. More generally, $d$-partitions correspond to so-called paving matroids of rank $d+1$.

A matroid is a finite set system $M = (E, \mathcal{H})$ which is a clutter (i.e.\ no element of $\mathcal{H}$ is properly contained in another) with the additional property that for every $x \in E$ and distinct $H_1, H_2 \in \mathcal{H}$, if $x \not \in H_1 \cup H_2$, then there exists $H_3 \in \mathcal{H}$ such that $(H_1 \cap H_2) \cup \{x\} \subseteq H_3$. The sets in $\mathcal{H}$ are called the hyperplanes of the matroid. The rank of a matroid is the cardinality of a smallest set that is not contained in a hyperplane.

Crapo and Rota~\cite{CrapoRota1970} observed that $d$-partitions always form the set of hyperplanes of a matroid of rank $d+1$ on the same ground set. Matroids that correspond to $d$-partitions in this way were first called ``paving'' by Welsh~\cite{Welsh1976}.

A matroid of rank~$r$ is called ``sparse paving'' if it is paving and each of its hyperplanes has cardinality $r-1$ or $r$; such matroids correspond to partial Steiner systems with parameters $(n,r-1,r)$. Sparse paving matroids play an important role in asymptotic enumeration of matroids, especially since it has been conjectured that almost every matroid is sparse paving (cf.\ \cite{CrapoRota1970,MayhewNewmanWelshWhittle2011,PendavinghVanderpol2015sparsepaving}).

The relations between the several notions related to paving matroids are depicted in Figure~\ref{fig:relations}. The results in this paper are framed in terms of matroids.

\begin{figure}[h]\centering
	\includegraphics[width=0.9\textwidth]{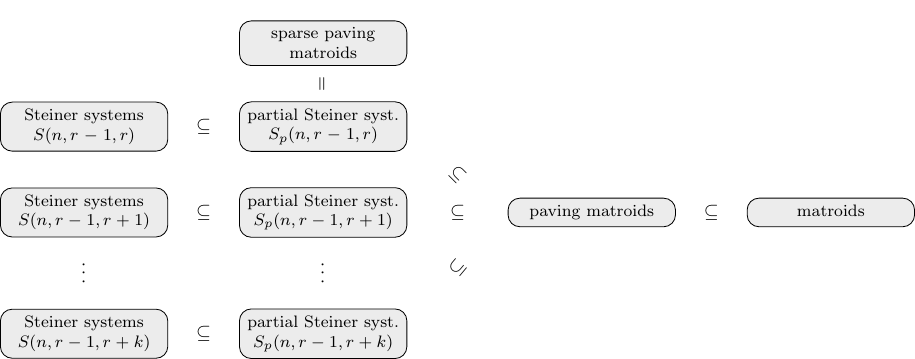}
	\caption{\label{fig:relations}The relations between the several classes of rank-$r$ matroids discussed in this paper.}
\end{figure}

\subsection{Asymptotic enumeration}

In addition to proving existence, Keevash obtained asymptotic estimates for the number of Steiner systems.
\begin{theorem}[\cite{Keevash2015}]\label{thm:keevash}
	The logarithm of the number of Steiner systems with parameters $(n,d,t)$, where they exist, is
	\begin{equation*}
		Q^{-1}\binom{n}{r}\bigl(\ln N + 1 - Q + o(1)\bigr)\qquad\text{ as $n\to\infty$},
	\end{equation*}
	where $N = \binom{n-d}{t-d}$ and $Q = \binom{t}{d}$.
\end{theorem}

In~\cite{PendavinghVanderpol2017}, two of the current authors obtained similar bounds for each of the classes of sparse paving, paving, and general matroids of fixed rank.
\begin{theorem}[\cite{PendavinghVanderpol2017}]\label{thm:oldbounds}
	Let $c(n,r)$ be the number of rank-$r$ matroids on a fixed ground set of cardinality $n$ in one of the classes sparse paving, paving, or general matroids. For fixed $r \ge 3$,
	\begin{equation*}
		\ln c(n,r) = \frac{1}{n}\binom{n}{r}\bigl(\ln(n) + \Theta(1)\bigr)
		\qquad\text{as $n\to\infty$}.
	\end{equation*}
\end{theorem}

In each of the cases, the $\Theta(1)$-term is at least $1-r+o(1)$ (which matches Keevash's result) and at most $1$.
In this paper, we show that it is in fact equal to $1-r+o(1)$ (at least for $r \ge 4$).

Write $s_k(n,r)$ for the number of paving matroids of rank $r$ on a fixed ground set of cardinality $n$ in which each hyperplane has cardinality $r-1$ or $r+k$.
\begin{theorem}\label{thm:sk(n,r)} For fixed $r \ge 3$ and $k \ge 0$,
	\begin{equation*}
		\ln s_k(n,r) \le Q^{-1}\binom{n}{r-1} \bigl(\ln N + 1 - Q + o(1)\bigr)\qquad\text{ as }n\to\infty,
	\end{equation*}
	where $N = \binom{n-r+1}{k+1}$ and $Q = \binom{r+k}{k+1}$.
\end{theorem}

We prove Theorem~\ref{thm:sk(n,r)} by an entropy counting method inspired by a method for counting Steiner triple systems due to Linial and Luria~\cite{LinialLuria2013}.

Let $s(n,r) = s_0(n,r)$ be the number of sparse paving matroids. Specialising Theorem~\ref{thm:sk(n,r)} to $k=0$, and combining it with the lower bound on $s(n,r)$ from~\cite{PendavinghVanderpol2017}), we obtain the following result.
\begin{theorem}\label{thm:s(n,r)}
	For fixed $r \ge 3$,
	$$\ln s(n,r) = \frac{1}{n-r+1} \binom{n}{r} \bigl(\ln(n-r+1) + 1 - r +o(1)\bigr)\qquad\text{ as }n\to\infty.$$
\end{theorem}

We next turn to enumeration of paving matroids, for which the following observation is crucial. Given any set $\mathcal{V}$ of $(d+1)$-subsets of $E$ (i.e.\ subsets $V\subseteq E$ with $|V|=d+1$), there is a unique maximum-cardinality $d$-partition $\mathcal{H}$ such that for each $V\in \mathcal{V}$ there is an $H\in \mathcal{H}$ with $V\subseteq H$. In turn, given any $d$-partition $\mathcal{H}$, it is not difficult to find some set $\mathcal{V}$ of $(d+1)$-subsets of $E$ which points to $\mathcal{H}$ in this manner. Thus we may encode $d$-partitions by sets of $(d+1)$-subsets of $E$, where the encodings may even be assumed to be of a special form.
To bound the number of $d$-partitions of $E$, it will then suffice to bound the number of  sets $\mathcal{V}$ of $(d+1)$-subsets of $E$ of this special form.

To bound the number of paving matroids $p(n,r)$ of rank $r\geq 4$, we argue that each paving matroid $M$ with hyperplanes $\mathcal{H}$ is encoded by a set of $r$-sets $\mathcal{V}$ which is the disjoint union of $r$-sets $\mathcal{V}^0$ and $r$-sets $\mathcal{V}^1$, such that $\mathcal{V}^0$ encodes the hyperplanes of a sparse paving matroid.
Exploiting a tradeoff between the cardinalities of $\mathcal{V}^0$ and $\mathcal{V}^1$ allows us to show that the number of paving matroids is close to the number of sparse paving matroids. 
\begin{theorem}\label{thm:p(n,r)}
	For fixed $r \ge 4$, $$\ln p(n,r) = \frac{1}{n-r+1} \binom{n}{r} \bigl(\ln(n-r+1) + 1 - r +o(1)\bigr)\qquad\text{ as }n\to\infty.$$
\end{theorem}

In rank $r=3$, the tradeoff between the cardinalities of $\mathcal{V}^0$ and $\mathcal{V}^1$ is not as significant as in higher ranks, and we resort to a different method. As noted, the sets $\mathcal{V}$ which we use to encode the hyperplanes $\mathcal{H}$ have a special form. We will derive bounds on the probability that a random set of triples from an $n$-set is {\em good} in this sense, and then bound $p(n,r)$ as the total number of sets of triples times this probability.
\begin{theorem}\label{thm:p(n,3)} For $r=3$,
	\begin{multline*}
		\frac{1}{n-r+1} \binom{n}{r} \bigl(\ln(n-r+1) - 2 + o(1)\bigr)
			\le \ln p(n,r) \\
			\le \frac{1}{n-r+1} \binom{n}{r} \bigl(\ln(n-r+1) + 0.35\bigr)
			\qquad\text{as $n\to\infty$}.
	\end{multline*}
\end{theorem}
The constant~0.35, which appears in the upper bound, is the result of an optimisation problem obtained from the combinatorial considerations outlined above, and is weaker than the pattern for higher ranks suggests. We do not know whether the upper bound is sharp, but speculate that it can be improved upon. See Section~\ref{sec:finalremarks} for further remarks.

\subsection{General matroids}

Write $m(n,r)$ for the number of matroids of rank~$r$ on a fixed ground set of cardinality~$n$. Perhaps surprisingly, our bounds on $p(n,r)$ imply similar bounds on $m(n,r)$. This follows from~\cite[Theorem~3]{PendavinghVanderpol2017}, which shows that $\ln p(n,r) \le \ln m(n,r) \le \left(1+\frac{r+o(1)}{n-r+1}\right)\ln p(n,r)$ for each fixed $r$. This observation, combined with Theorem~\ref{thm:p(n,r)}, thus immediately implies the following result on the number of matroids.
\begin{theorem}\label{thm:m(n,r)}
	For fixed $r \ge 4$,
	\begin{equation*}
		\ln m(n,r) = \frac{1}{n-r+1} \binom{n}{r} \bigl(\ln(n-r+1) + 1 - r + o(1)\bigr)
		\qquad\text{as $n\to\infty$}.
	\end{equation*}
\end{theorem}

\subsection{The remainder of the paper}

In the next section, we describe some preliminaries, after which the remainder of the paper is subdivided according to the methods used.
In Section~\ref{sec:designs}, we use entropy methods to bound the number of partial designs and sparse paving matroids.
In Section~\ref{sec:p(n,r)}, we describe the encoding of the hyperplanes of a paving matroid (as outlined above) and establish the bounds in rank $r\ge 4$; that section uses elementary combinatorial counting arguments.
In Section~\ref{sec:p(n,3)}, we use probabilistic arguments and continuous optimisation to prove the upper bound in rank $r=3$.
In the final section, we speculate on the remaining gap between upper and lower bounds in the rank-3 case.

\section{Preliminaries}

Throughout this paper, we use $[n]$ as a shorthand for the set $\{1,2,\ldots,n\}$. We write $\PP(n,r)$ and $\SS(n,r)$ for the sets of paving and sparse paving matroids, respectively, of rank $r$ on ground set $[n]$. In addition, we use $p(n,r) = |\PP(n,r)|$ and $s(n,r) = |\SS(n,r)|$.

If $E$ is a finite set, and $0 \le r \le |E|$, then we write
\begin{equation*}
	\binom{E}{r} \coloneqq \left\{X \subseteq E : |X|=r\right\}.
\end{equation*}

Given a collection of sets $\mathcal{X} \subseteq 2^{E}$, the $s$-shadow $\partial_s \mathcal{X}$ of $\mathcal{X}$ is
\begin{equation*}
	\partial_s \mathcal{X} \coloneqq \left\{ Y \in \binom{E}{s} : \text{there exists $X \in \mathcal{X}$ such that $Y \subseteq X$}\right\}.
\end{equation*}

The following bounds, which are valid for all integers $k \ge 1$, are a form of Stirling's approximation:
\begin{equation}\label{eq:stirling}
	\sqrt{2\pi k} \left(\frac{k}{\e}\right)^k
		\le k!
		\le \e\sqrt{k} \left(\frac{k}{\e}\right)^k.
\end{equation} 

We freely use the standard bound on sums of binomial coefficients
\begin{equation}\label{eq:prelim:binomial-ub}
	\sum_{i=0}^{m} \binom{n}{k} \le \left(\frac{\e n}{m}\right)^m,
\end{equation}
which is valid for all $0 < m \le n$.

The following lemma provides a bound in the other direction. It essentially shows that the constant $\e$ that appears in the upper bound cannot be dispensed with. 

\begin{lemma}\label{lemma:prelim:binomial-lb}
	For all $1 \le k \le n$, $\binom{n}{k} \ge \left(\frac{\e^{1-\varepsilon}n}{k}\right)^k$, where $\varepsilon\!\equiv\!\varepsilon_{k,n}\!=\!\frac{1}{k} \ln\frac{\e\sqrt{k}}{\prod_{i=0}^{k-1} \left(1-i/n\right)}$.
\end{lemma}

\begin{proof}
	Since $\binom{n}{k} = \frac{n\dotsm(n-k+1)}{k!}$, it follows from~\eqref{eq:stirling} that
	\begin{equation*}
		\binom{n}{k} \ge \left(\frac{\e n}{k}\right)^k \frac{\prod_{i=0}^{k-1} \left(1-\frac{i}{n}\right)}{\e\sqrt{k}}. \qedhere
	\end{equation*}
\end{proof}

\section{\label{sec:designs}Sparse paving matroids}

In this section, we prove Theorems~\ref{thm:sk(n,r)} and~\ref{thm:s(n,r)}. The upper bound in Theorem~\ref{thm:s(n,r)} is obtained by specialising that in Theorem~\ref{thm:sk(n,r)} to $k=0$, while the lower bound was proved in~\cite{PendavinghVanderpol2017}:
\begin{proposition}[{\cite[Theorem~10]{PendavinghVanderpol2017}}]\label{prop:s(n,r):lower-bound}
	For each $r \ge 3$,
	\begin{equation*}
		\ln s(n,r) \ge \frac{1}{n-r+1}\binom{n}{r}\bigl(\ln(n-r+1) + 1 - r +o(1)\bigr)
		\qquad\text{as $n\to\infty$}.
	\end{equation*}
\end{proposition}
The remainder of this section is devoted to the proof of Theorem~\ref{thm:sk(n,r)}. The upper bound in that theorem is proved using information-theoretic techniques; we review some of the notation and terminology that we require in Section~\ref{ss:entropy}. Subsequently, in Section~\ref{ss:sk(n,r)}, we bound the number of partial Steiner triple systems whose blocks have a given shadow. By summing over all possible shadows, this result implies Theorem~\ref{thm:sk(n,r)}.

\subsection{Entropy}\label{ss:entropy}

The upper bound on $s(n,r)$ is proved using information-theoretic techniques. We review some of the notation and terminology that we require; for a more thorough introduction, we refer the reader to~\cite[Section~15.7]{AlonSpencer2008}.

In what follows, bold-faced symbols, such as $\rv{X}$, are random variables that take their values in some finite set $\mathcal{X}$ and $\pprob$ denotes their law. The entropy $\entropy{\rv{X}}$ of $\rv{X}$ is defined as
\begin{equation*}
	\entropy{\rv{X}} \coloneqq -\sum_{x \in \mathcal{X}} \prob{\rv{X}=x} \ln\prob{\rv{X}=x},
\end{equation*}
where for convenience we use $0 \ln 0 = 0$.

It is always true that $\entropy{\rv{X}} \le \ln |\mathcal{X}|$. The upper bound is attained if (and only if) $\rv{X}$ has the uniform distribution on $\mathcal{X}$. This observation makes entropy useful for enumeration purposes: questions about the cardinality of $\mathcal{X}$ immediately translate to questions about the entropy of random variables having the uniform distribution on $\mathcal{X}$.

For a pair of random variables $(\rv{X}, \rv{Y})$, the conditional entropy of $\rv{X}$ given $\rv{Y}$ is
\begin{equation*}
	\condentropy{\rv{X}}{\rv{Y}} =
		- \sum_{y} \prob{\rv{Y}=y} \sum_{x} \condprob{\rv{X} = x}{\rv{Y} = y} \ln \condprob{\rv{X}=x}{\rv{Y}=y},
\end{equation*}
which can be written as $\condentropy{\rv{X}}{\rv{Y}} = \entropy{\rv{X}, \rv{Y}} - \entropy{\rv{Y}}$. More generally, if $\rv{X} = (\rv{X}_1, \ldots, \rv{X}_n)$ is a sequence of random variables, then the chain rule for entropy states that
\begin{equation*}
	\entropy{\rv{X}} = \entropy{\rv{X}_1} + \condentropy{\rv{X}_2}{\rv{X}_1} + \ldots + \condentropy{\rv{X}_n}{\rv{X}_1, \ldots, \rv{X}_{n-1}}.
\end{equation*}

\subsection{Upper bound}\label{ss:sk(n,r)}

Let $\SS_k(n,r) \subseteq \PP(n,r)$ be the collection of paving matroids all of whose hyperplanes have cardinality $r-1$ or $r+k$ (the hyperplanes of cardinality $r+k$ of such a matroid form a partial Steiner system on $n$ points, in which each block has cardinality $r+k$ and each $(r-1)$-set is contained in at most one block). Note that $\SS(n,r) = \SS_0(n,r)$. Partition $\SS_k(n,r)$ according to the $(r-1)$-shadows of dependent hyperplanes: for a matroid~$M$, let $\HH_k(M)$ be the collection of its hyperplanes of cardinality~$r+k$, and for $\mathcal{A} \in \binom{[n]}{r-1}$, write
\begin{equation*}
	\SS_k(n,r,\mathcal{A}) = \left\{M \in \SS_k(n,r) : \partial_{r-1}\HH_k(M) = \mathcal{A}\right\},
\end{equation*}
and let $s_k(n,r,\mathcal{A}) = |\SS_k(n,r,\mathcal{A})|$ and $s(n,r,\mathcal{A}) = s_0(n,r,\mathcal{A})$. (Note that $\SS_k(n,r,\mathcal{A})$ may be empty for some choices of $\mathcal{A}$, but this is immaterial to our argument.)

The next lemma is a generalisation to partial Steiner systems of a result of Linial and Luria~\cite{LinialLuria2013} for Steiner triple systems (their result was generalised to arbitrary designs by Keevash in~\cite[Theorem~6.1]{Keevash2015}).

\begin{lemma}\label{lemma:p(n,r,k,A)-bound}
	For each $r \ge 3$, and $k \ge 0$, there exists a function $f^{\text{(\ref{lemma:p(n,r,k,A)-bound})}}_{r,k}(n)$ with the property that $f^{\text{(\ref{lemma:p(n,r,k,A)-bound})}}_{r,k}(n) \to 0$ as $n \to \infty$, such that
	\begin{equation*}
		\ln s_k(n,r,\mathcal{A}) \le \frac{|\mathcal{A}|}{Q}(\ln N + 1 - Q + f^{\text{(\ref{lemma:p(n,r,k,A)-bound})}}_{r,k}(n))
	\end{equation*}
	for all $\mathcal{A} \subseteq \binom{[n]}{r-1}$, where $Q = \binom{r+k}{k+1}$ and $N = \binom{n-r+1}{k+1}$. In particular,
	\begin{equation*}
		\ln s(n,r,\mathcal{A}) \le \frac{|\mathcal{A}|}{r}(\ln(n-r+1) + 1 - r + f^{\text{(\ref{lemma:p(n,r,k,A)-bound})}}_{r,0}(n))
	\end{equation*}
	for all $\mathcal{A} \subseteq \binom{[n]}{r-1}$.
\end{lemma}

\begin{proof}
	Fix $\mathcal{A}$ and let $\rv{X}$ be a matroid chosen uniformly at random from $\SS_k(n,r,\mathcal{A})$. As $\ln s_k(n,r,\mathcal{A}) = \entropy{\rv{X}}$, it suffices to bound $\entropy{\rv{X}}$.
	
	Consider the collection of random variables $\left\{\rv{X}_A : A \in \mathcal{A}\right\}$, where $\rv{X}_A$ is the unique hyperplane of $\rv{X}$ containing $A$, and note that $\entropy{\rv{X}} = \entropy{\rv{X}_A : A \in \mathcal{A}}$.
	Order the collection~$\mathcal{A}$. This is conveniently done by introducing an injective function $\lambda\colon\mathcal{A}\to[0,1]$ and ordering $\mathcal{A}$ by decreasing $\lambda$-values.
	Write $\Xhist{A} \coloneqq (\rv{X}_{A'} : \lambda(A') > \lambda(A))$. By the chain rule for entropy,
	\begin{equation*}
		\entropy{\rv{X}} = \sum_{A \in \mathcal{A}} \condentropy{\rv{X}_A}{\Xhist{A}}.
	\end{equation*}
	
	For $A \in \mathcal{A}$, let
	\begin{equation*}
		\mathcal{X}_A \coloneqq \left\{X \in \binom{[n]}{r+k} : \begin{matrix*}[l]A \subseteq X\text{, and} \\ \text{$A' \in \mathcal{A}$ for all $A' \in \binom{X}{r-1}$}\end{matrix*}\right\}.
	\end{equation*}
	Clearly, $\mathcal{X}_A$ depends only on $A$ and $\mathcal{A}$. Note that $\rv{X}_A \in \mathcal{X}_A$ and $1 \le |\mathcal{X}_A| \le \binom{n-r+1}{k+1}$.	
	We further restrict the number of possible values for $\rv{X}_A$ conditional on $\Xhist{A}$. If $A \subseteq X$ for some $X \in \Xhist{A}$, then we must have $\rv{X}_A = X$. On the other hand if $A$ is not contained in any member of $\Xhist{A}$, then in order for $H \in \mathcal{X}_A$ to be available for~$\rv{X}_A$, we cannot have $\rv{X}_{A''} \in \Xhist{A}$ for any $A'' \in \binom{\rv{X}_{A'}}{r-1}$, where $A' \in \binom{H}{r-1}\setminus\{A\}$. We make this precise by introducing the events
	\begin{equation*}\label{eq:entropy:event}
		\mathcal{E}_{H,A} \coloneqq
		\bigcap_{A' \in \binom{H}{r-1}\setminus\{A\}}\;
		\bigcap_{A'' \in \binom{X_{A'}}{r-1}}\;
		\bigcap_{X \in \rv{X}_A^{\lambda}}
		\left\{A'' \not\subseteq X\right\}
	\end{equation*}
	and writing
	\begin{equation}\label{eq:p(n,r,k,A):NA-def}
		\rv{N}_A \coloneqq
		\begin{cases}
			\hphantom{\sum\limits_{H \in \mathcal{X}_A}}
			1
				& \text{if $\exists X \in \Xhist{A}: A \subseteq X$} \\
			\sum\limits_{H \in \mathcal{X}_A}
			\indi{\mathcal{E}_{H,A}}
				& \text{otherwise},
		\end{cases}
	\end{equation}
	so that, by the above discussion, $\condentropy{\rv{X}_A}{\Xhist{A}} \le \expect[\rv{X}]{\ln \rv{N}_A}$. This inequality holds for any injection~$\lambda$, so it remains true after randomising~$\lambda$ and taking the expected value. Such a random~$\rv{\lambda}$ can be constructed by choosing~$\rv{\lambda}(A)$ uniformly at random from the interval~$[0,1]$, independently of all other choices and of~$\rv{X}$. (Note that almost surely no two $\rv{\lambda}$-values are the same.) We obtain
	\begin{equation*}
		\expect[\rv{X}]{\ln \rv{N}_A} =
		\expect[\rv{\lambda}(A)]{
			\expect[\rv{X}]{
				\condexpect[\rv{\lambda}]{\ln \rv{N}_A}{\rv{\lambda}(A)}
			}
		}.
	\end{equation*}
	
	Let $\mathcal{F}_A$ be the event that $\rv{\lambda}(A) > \rv{\lambda}(A')$ for all $A' \in \binom{\rv{X}_A}{r-1}\setminus\{A\}$, i.e.\ that $A$ comes first (in the $\rv{\lambda}$-ordering) among all $(r-1)$-subsets of $\rv{X}_A$. Using that $\rv{N}_A = 1$ on $\overline{\mathcal{F}_A}$, we obtain
	\begin{equation*}
		\begin{split}
		{\condexpect[\rv{\lambda}]{\ln \rv{N}_A}{\rv{\lambda}(A)}}
			&= \left(\rv{\lambda}(A)\right)^{Q-1} {\condexpect[\rv{\lambda}]{\ln \rv{N}_A}{\rv{\lambda}(A), \mathcal{F}_A}} \\
			&\le \left(\rv{\lambda}(A)\right)^{Q-1} \ln {\condexpect[\rv{\lambda}]{\rv{N}_A}{\rv{\lambda}(A), \mathcal{F}_A}},
		\end{split}
	\end{equation*}
	where the inequality follows from Jensen's inequality.
	
	We claim that
	\begin{equation}\label{eq:p(n,r,k,A)-bound:1}
		\expect[\rv{\lambda}]{\rv{N}_A \mid \rv{\lambda}(A), \mathcal{F}_A} \le 1 + N (\rv{\lambda}(A))^{Q(Q-1)} + NQ\frac{8(k+1)^2}{n},
	\end{equation}
	provided that~$n$ is sufficiently large.
	To prove~\eqref{eq:p(n,r,k,A)-bound:1}, first note that on the event $\mathcal{F}_A$ the term in~\eqref{eq:p(n,r,k,A):NA-def} corresponding to $\rv{X}_A \in \mathcal{F}_A$ evaluates to 1.
	Call $H \in \mathcal{X}_A$ good if $H \cap \rv{X}_{A'} = A'$ for all $A' \in \binom{H}{r-1}\setminus\{A\}$, and bad otherwise. We show that the second term in the right-hand side of~\eqref{eq:p(n,r,k,A)-bound:1} bounds the sum over all terms corresponding to good~$H$ in~\eqref{eq:p(n,r,k,A):NA-def}, while the final term bounds the sum over bad~$H$.
	
	For good $H \in \mathcal{X}_A\setminus\{A\}$, the event $\mathcal{E}_{H,A}$ happens precisely when $A$ precedes (in the $\rv\lambda$-ordering) all $Q(Q-1)$ of the $(r-1)$-sets contained in a set of the form $\rv{X}_{A'}$ with $A' \in \binom{H}{r-1}\setminus\{A\}$. As these events are mutually independent (since~$H$ is good), and each happens with probability $\rv{\lambda}(A)$, the second term in the right-hand side of~\eqref{eq:p(n,r,k,A)-bound:1} follows by linearity of expectation.
	
	 For bad $H$, we use the trivial upper bound on the probability of~$\mathcal{E}_{H,A}$. It remains to show that the number of bad $H$ is at most $NQ\frac{8(k+1)^2}{n}$. Let $\mathcal{X}'_A \coloneqq \left\{X \in \binom{[n]}{r+k} : \binom{X}{r-1} \subseteq \mathcal{A}\right\}$, and note that the number of bad $H$ is bounded from above by the number of bad $H$ in $\mathcal{X}'_A$. Pick ordered $\overrightarrow{H} = (h_1, \ldots, h_{r+k}) \in \mathcal{X}'_A$ uniformly at random. Write $H = \{h_1, \ldots, h_{r+k}\}$, and for $J \subseteq [r+k]$ write $H(J) = \{h_j : j \in J\}$. By conditioning on $H(J)$ for $J \in \binom{[r+k]}{r-1}$ we obtain
	\begin{equation*}
		\prob[\overrightarrow{H}]{\left|\left(\rv{X}_{H(J)} - H(J)\right)\cap\left(\rv{H}-H(J)\right)\right| \ge 1 \:\Big\vert\: H(J)}
			= 1 - \frac{\binom{n-r-k}{k+1}}{\binom{n-r+1}{k+1}},
	\end{equation*}
	which is at most $\frac{4(k+1)^2}{n}$.
	As there are $Q$ such sets $J$, the union bound implies that the fraction of bad elements of $\mathcal{X}'_A$ is at most~$Q\frac{4(k+1)^2}{n}$.
	The bound on the number of bad $H$ follows since $|\mathcal{X}'_A| \le 2N$ for sufficiently large~$n$.
	
	Now that we have established~\eqref{eq:p(n,r,k,A)-bound:1}, we conclude that
	\begin{equation}\label{eq:p(n,r,k,A)-bound:2}
		\begin{split}
			\entropy{\rv{X}}
				&= \sum_{A \in \mathcal{A}} \int_0^1 \lambda^{Q-1} \ln\left(1 + N\lambda^{Q(Q-1)} + N\frac{8Q(k+1)^2}{n}\right)\dd{\lambda} \\
				&= \frac{|\mathcal{A}|}{Q} \int_0^1 \ln\left(1+Nu^{Q-1} + N\frac{8Q(k+1)^2}{n}\right) \dd{u} \\
				&\le \frac{|\mathcal{A}|}{Q}\left[\ln N + \int_0^1 \ln\left(u^{Q-1} + \frac{8Q(k+1)^2}{n} + \frac{1}{N}\right)\dd{u}\right].
		\end{split}
	\end{equation}
	Let $f^{\text{(\ref{lemma:p(n,r,k,A)-bound})}}_{r,k}(n) = 3Q\left(\frac{8Q(k+1)^2}{n} + \frac{1}{N}\right)^{\frac{1}{Q-1}}$. The integral on the right-hand side of~\eqref{eq:p(n,r,k,A)-bound:2} is at most~$1-Q+f^{\text{(\ref{lemma:p(n,r,k,A)-bound})}}_{r,k}(n)$, which proves the first claim.
	The second claim follows from the first, since $s(n,r,\mathcal{A}) = s_0(n,r,\mathcal{A})$.
\end{proof}

The following lemma bounds the number of partial designs with given parameters.
In particular, it proves Theorem~\ref{thm:sk(n,r)}.
\begin{lemma}\label{lemma:partial-designs}
	For each $r \ge 3$ and $k \ge 0$, there exists a function $f^{\text{(\ref{lemma:partial-designs})}}_{r,k}(n)$ with the property that $f^{\text{(\ref{lemma:partial-designs})}}_{r,k}(n) \to 0$ as $n \to \infty$, such that
	\begin{equation*}
		\ln s_k(n,r) \le Q^{-1}\binom{n}{r-1} \left(\ln N + 1 - Q + f^{\text{(\ref{lemma:partial-designs})}}_{r,k}(n)\right),
	\end{equation*}
	where $N = \binom{n-r+1}{k+1}$ and $Q = \binom{r+k}{k+1}$.
\end{lemma}

\begin{proof}
	Let $f^{\text{(\ref{lemma:partial-designs})}}_{r,k}(n) = f^{\text{(\ref{lemma:p(n,r,k,A)-bound})}}_{r,k}(n) + Q \ln\left(1+\left(N \e^{1-Q+f^{\text{(\ref{lemma:p(n,r,k,A)-bound})}}_{r,k}(n)}\right)^{-Q^{-1}}\right)$.
	A straightforward argument shows that $f^{\text{(\ref{lemma:partial-designs})}}_{r,k}(n) \to 0$ as $n \to \infty$.
	As $s_k(n,r) = \sum_{\mathcal{A}} s_k(n,r,\mathcal{A})$, where the sum is over all subsets $\mathcal{A} \subseteq \binom{[n]}{r-1}$, it follows from Lemma~\ref{lemma:p(n,r,k,A)-bound} that
	\begin{multline*}
		s_k(n,r)
			\le \sum_{a=0}^{\binom{n}{r-1}} \binom{\binom{n}{r-1}}{a} \left(\e^{1-Q+f^{\text{(\ref{lemma:p(n,r,k,A)-bound})}}_{r,k}(n)}N\right)^{Q^{-1}a} \\
			= \left(1+\left(\e^{1-Q+f^{\text{(\ref{lemma:p(n,r,k,A)-bound})}}_{r,k}(n)}N\right)^{Q^{-1}}\right)^{\binom{n}{r-1}}
			= \left(\e^{1-Q+f^{\text{(\ref{lemma:partial-designs})}}_{r,k}(n)}N\right)^{Q^{-1}\binom{n}{r-1}},
	\end{multline*}
	as required.
\end{proof}

\section{\label{sec:p(n,r)}Paving matroids of rank at least 4}

In this section, we prove Theorem~\ref{thm:p(n,r)}.

\subsection{An encoding of paving matroids}
\label{subsec:p(n,r):encoding}

We describe an encoding of paving matroids that was used in~\cite{PendavinghVanderpol2017} to prove a weaker bound on the number of paving matroids.

Let $E$ be a finite set and assume that it is linearly ordered.
A paving matroid  $M$ of rank $r$ on $E$ can be reconstructed from the collection
\begin{equation*}
	\VV(M)\coloneqq\!\!\bigcup_{H\in \mathcal{H}(M)}\!\! \VV(H),
\end{equation*}
where for each hyperplane $H$, the elements of $\VV(H)$ are exactly the  consecutive $r$-subsets of $H$:
\begin{equation*}
	\VV(H)\coloneqq\left\{V\in \binom{H}{r}: \text{there are no }v,v'\in V\text{ and }h\in H\setminus V\text{ so that }v<h<v'\right\}.
\end{equation*}
Given $V \in \VV(M)$, the dependent hyperplane of $M$ spanned by $V$ is the inclusionwise minimal set $H \supseteq V$ with the property that for all $V' \in \VV(M)$, if $|V' \cap H| \ge r-1$, then $V' \subseteq H$; as each dependent hyperplane is spanned by some $V \in \VV(M)$, the matroid $M$ can be reconstructed from $\VV(M)$.

If $M$ is a sparse paving matroid, then $\VV(M)$ is the collection of circuit-hyperplanes of $M$, and hence a stable set in the Johnson graph $J(E,r)$. In general, this is not the case, as $\VV(M)$ may contain
distinct sets $V, V'$ for which $|V\cap V'|=r-1$.
The occurrence of such sets, however, is relatively restricted; 
this is the content of the following two lemmas, whose proofs can be found in~\cite[Section~4.2]{PendavinghVanderpol2017} as well as~\cite[Lemmas~5.5.3 and~5.5.5]{VanderPol2017}. We note that Lemma~\ref{lemma:p(n,r):encoding-stable} follows immediately from the observation that each member of $\VV(H)$ is an $r$-element circuit of $M$, while Lemma~\ref{lemma:p(n,r):encoding-small} can be proved using an argument similar to that used to prove Lemma~\ref{lemma:v0v1} below.
\begin{lemma}\label{lemma:p(n,r):encoding-stable}
	Let $M$ be a paving matroid of rank $r$ and let $H, H'$ be distinct hyperplanes of $M$. If $V\in \VV(H)$ and $V'\in \VV(H')$, then $|V\cap V'|< r-1$.
\end{lemma}

\begin{lemma}\label{lemma:p(n,r):encoding-small}
	If $M$ is a paving matroid of rank~$r$ on a ground set of cardinality $n$, then $|\VV(M)| \le \frac{1}{n-r+1}\binom{n}{r}$.
\end{lemma}
For each hyperplane $H$, we may write $\VV(H)=\{V_H^0,\ldots, V_H^k\}$ so that  $|V_H^i\cap V_H^j|=r-1$ if and only if $i=j \pm 1$. Consider the partition $\VV(H)=\VV^{0}(H)\dot{\cup} \VV^{1}(H)$ where 
\begin{equation*}
	\VV^{0}(H)=\{V_H^i: i\text{ even} \}
	\qquad\text{and}\qquad
	\VV^{1}(H)=\{V_H^i: i\text{ odd} \}.
\end{equation*}
The collections $\VV^{0}(H)$ and $\VV^{1}(H)$ are both stable sets of $J(E,r)$. Writing
\begin{equation*}
	\VV^{0}(M)\coloneqq\!\!\bigcup_{H\in \mathcal{H}(M)}\!\! \VV^{0}(H)
	\qquad\text{and}\qquad
	\VV^{1}(M)\coloneqq\!\!\bigcup_{H\in \mathcal{H}(M)}\!\! \VV^{1}(H),
\end{equation*}
we evidently have $\VV(M)=\VV^{0}(M)\dot{\cup}\VV^{1}(M)$. Moreover, by Lemma~\ref{lemma:p(n,r):encoding-stable} both $\VV^{0}(M)$ and $\VV^{1}(M)$ are stable sets of $J(E,r)$ and hence each forms the collection of circuit-hyperplanes of a rank-$r$ sparse paving matroid on $E$.

We associate two $(r-1)$-shadows with a paving matroid $M$: $\partial_{r-1} \VV^0(M)$ and $\partial_{r-1} \VV^1(M)$. In the remainder of this section, all shadows will be $(r-1)$-shadows, and we suppress the subscript $r-1$ in our notation.

\begin{lemma}\label{lemma:v0v1}
	Let $n \ge r \ge 3$. For each $M \in \PP(n,r)$,
	\begin{equation*}
		|\partial \VV^0(M)| + \frac{r-1}{2}|\partial \VV^1(M)| \le \binom{n}{r-1}.
	\end{equation*}
\end{lemma}
\begin{proof}
	For $k\ge-1$, let $h_k$ denote the number of hyperplanes of $M$ that contain exactly $r+k$ elements. As each $(r-1)$-set from $E(M)$ is contained in a unique hyperplane, we have
\begin{equation}\label{hyp}\sum_{k=-1}^\infty h_k\binom{r+k}{r-1} = \binom{n}{r-1}.\end{equation}
Each hyperplane with $r+k$ elements contributes $\lfloor k/2\rfloor+1$ elements to $\VV^0(M)$ and $\lceil k/2\rceil$ elements to $\VV^1(M)$. Hence, writing $v_i = |\VV^i(M)|$, $i \in \{0,1\}$,
\begin{equation*}\label{hyp2}
	v_0=\sum_{k=-1}^\infty h_k(\lfloor k/2\rfloor+1)
	\quad\text{and}\quad
	v_1=\sum_{k=-1}^\infty h_k\lceil k/2\rceil.
\end{equation*}
As $r(\lfloor k/2\rfloor+1)+r\frac{r-1}{2}\lceil k/2\rceil\leq \binom{r+k}{r-1}$ for all $k\ge-1$ (which is easily shown for $k=-1$ and $k=0$ and follows by induction and the inequality $\binom{r+k+1}{r-1} \ge \binom{r+k}{r-1} + \binom{r}{r-2}$ for larger $k$), it follows from \eqref{hyp} that
\begin{equation*}
	\begin{split}
		r(v_0+\frac{r-1}{2}v_1)
			&\le \sum_{k=-1}^\infty h_k\left(r(\lfloor k/2\rfloor+1)+r\frac{r-1}{2}\lceil k/2\rceil\right) \\
			&\le \sum_{k=-1}^\infty h_k\binom{r+k}{r-1}
			= \binom{n}{r-1}.
	\end{split}
\end{equation*}
The lemma follows, since $|\partial \VV^i(M)| = r v_i$ for $i \in \{0,1\}$.
\end{proof}

\subsection{Upper bound}
We now turn to proving the upper bound in Theorem~\ref{thm:p(n,r)}.
Define
$$\mathbb{P}(n,r,\mathcal{A}, \mathcal{B})\coloneqq\left\{M\in \mathbb{P}(n,r): \partial \VV^0(M)=\mathcal{A}, \partial \VV^1(M)=\mathcal{B}\right\}$$
and $p(n,r,\mathcal{A}, \mathcal{B})\coloneqq|\mathbb{P}(n,r,\mathcal{A}, \mathcal{B})|$.
We consider that 
$$p(n,r) = \sum_{a=0}^{\binom{n}{r-1}} \sum_{b=0}^{\binom{n}{r-1}} p(n,r,a,b),$$ 
where $p(n,r,a,b)$ denotes the sum of $p(n,r,\mathcal{A}, \mathcal{B})$ over all $\mathcal{A}, \mathcal{B}\subseteq\binom{[n]}{r}$ such that $|\mathcal{A}|=a, |\mathcal{B}|=b$.
Note that for $p(n,r,a,b) > 0$ to hold, both $a$ and $b$ are necessarily multiples of $r$.
We prove sufficient bounds on $\ln p(n,r,a,b)$ under two complementary  regimes. 
\begin{lemma} \label{lem:paving_a_small} Let $r\geq 4$ and $n \ge \exp\left(\left(\frac{(r-1)(r+1)}{r-3}\right)^2\right)$. If $a\leq \Big(1-\frac{1}{\sqrt{\ln(n)}}\Big)\binom{n}{r-1}$, then
\begin{equation}\label{eq:paving_a_small}
	\ln p(n,r,a,b)\leq \frac{1}{r}\binom{n}{r-1}\left(\ln(n-r+1)+1-r\right).
\end{equation}
\end{lemma}
\proof Suppose $\mathcal{A}, \mathcal{B}\subseteq \binom{[n]}{r-1}$ are such that $|\mathcal{A}|=a$ and $|\mathcal{B}|=b$. Each $M\in \mathbb{P}(n,r,\mathcal{A}, \mathcal{B})$ is determined uniquely by $\VV(M)=\VV^0(M)\cup\VV^1(M)$, and we have $\partial \VV^0(M)=\mathcal{A}$ and $\partial \VV^1(M)=\mathcal{B}$. Let $\gamma_r = \frac{r-1}{2}$. By Lemma \ref{lemma:v0v1},
$$a+\gamma_r b=|\partial \VV^0(M)|+\gamma_r |\partial \VV^1 (M)|\leq \binom{n}{r-1}.$$ 
Since  $|\VV(M)|=|\VV^0(M)|+|\VV^1(M)|=(a+b)/r$, and using the assumed upper bound on $a$, we have 
\begin{equation*}
	|\VV(M)|
		\leq \frac{1}{r}\binom{n}{r-1}\left(1-\frac{\delta_r}{\sqrt{\ln(n)}} \right),
\end{equation*}
where $\delta_r=1-\gamma_r^{-1} = 1-\frac{2}{r-1} >0$.
As each of the paving matroids $M$ we are counting is determined uniquely by a set $\VV(M)\subseteq \binom{[n]}{r}$ of this bounded cardinality, we obtain
\begin{equation*}
	\begin{split}
		\ln p(n,r,a,b)&\le   \frac{1}{r}\binom{n}{r-1}\left(1-\frac{\delta_r}{\sqrt{\ln(n)}} \right)\ln\left(\frac{\e(n-r+1)}{1-\frac{\delta_r}{\sqrt{\ln(n)}}}\right) \\
		&=   \frac{1}{r}\binom{n}{r-1}\left(\ln(n-r+1)+1+u_r(n)\right),
	\end{split}
\end{equation*}
where $u_r(n)=-\left(1-\frac{\delta_r}{\sqrt{\ln n}}\right)\ln\left(1-\frac{\delta_r}{\sqrt{\ln n}}\right)-\frac{\delta_r}{\sqrt{\ln n}} - \frac{\delta_r}{\sqrt{\ln n}} \ln(n-r+1)$.
It is straightforward to verify that $u_r(n) \le -r$ whenever $n \ge \exp((r+1)^2\delta_r^{-2})$, which completes the proof of~\eqref{eq:paving_a_small}.
\endproof
\begin{lemma}\label{lem:paving_a_large}For each $r\ge 4$, there exists a function $f_r^{\text{(\ref{lem:paving_a_large})}}(n)$ with the property that $f_r^{\text{(\ref{lem:paving_a_large})}}(n)\rightarrow 0$ as $n\rightarrow \infty$ such that if $a\geq \left(1-\frac{1}{\sqrt{\ln(n)}}\right)\binom{n}{r-1}$, then 
\begin{equation}\label{eq:paving_a_large}
	\ln p(n,r,a,b)
		\le \frac{1}{r}\binom{n}{r-1}\left(\ln(n-r+1)+1-r+f_r^{\text{(\ref{lem:paving_a_large})}}(n)\right).
\end{equation}
\end{lemma}
\begin{proof}
Fix $r \ge 4$, let $\gamma_r = \frac{r-1}{2}$, and let $n_0 \equiv n_0(r)$ be so large that
\begin{equation}\label{eq:paving_a_large:n0}
	\ln(\e(n-r+1)\ln^2(n)) \le \gamma_r\ln\left(\e^{1-r}(n-r+1)\right)\qquad\text{for all $n \ge n_0$}.
\end{equation}
Let $f_r^{\text{(\ref{lem:paving_a_large})}}(n)$ be a function such that
\begin{equation*}
	f_r^{\text{(\ref{lem:paving_a_large})}}(n)
		= f_{r,0}^{(\ref{lemma:p(n,r,k,A)-bound})}(n)
			+ r\frac{\ln(\e\sqrt{\ln(n)})}{\sqrt{\ln(n)}}
			+ \frac{3}{\ln(n)}
\end{equation*}
for all $n \ge n_0$ and $f_r^{\text{(\ref{lem:paving_a_large})}}(n) = C$ when $n < n_0$, where $C$ is a large constant which may depend on $r$. It is clear that $f_r^{\text{(\ref{lem:paving_a_large})}}(n) \to 0$, while~\eqref{eq:paving_a_large} holds for $n \le n_0$, provided $C$ is sufficiently large. We may therefore assume that $n \ge n_0$.

Each $M\in \mathbb{P}(n,r,\mathcal{A}, \mathcal{B})$ is determined uniquely by the pair $(\VV^0(M), \VV^1(M))$, which is such that $\partial \VV^0(M)=\mathcal{A}$ and $\partial \VV^1(M)=\mathcal{B}$. The collection $\VV^0(M)$ is the collection of circuit-hyperplanes of a sparse paving matroid $N \in \SS(n,r,\mathcal{A})$,
and $\VV^1(M)$ is a collection of $|\mathcal{B}|/r$ elements from $\binom{[n]}{r}$, which in turn determines $\mathcal{B}$. For fixed $\mathcal{A}$ and $b$, it follows that
\begin{equation*}
	\sum_{\mathcal{B}: |\mathcal{B}|=b}\!\!\! p(n,r,\mathcal{A},\mathcal{B})
		\le\!\!\!\sum_{\mathcal{U}: |\mathcal{U}|=b/r} \!\!\!|\{M\in \mathbb{P}(n,r,\mathcal{A}, \partial \mathcal{U}): \VV^1(M)=\mathcal{U}\}|
		\leq s(n,r,\mathcal{A})\binom{\binom{n}{r}}{b/r}.
\end{equation*}
By definition of $p(n,r,a,b)$, we have
$$p(n,r,a,b)=\sum_{\mathcal{A}: |\mathcal{A}|=a}\sum_{\mathcal{B}: |\mathcal{B}|=b} p(n,r,\mathcal{A},\mathcal{B})\leq \sum_{\mathcal{A}: |\mathcal{A}|=a}s(n,r,\mathcal{A})\binom{\binom{n}{r}}{b/r}.$$
Using the upper bound on $\ln s(n,r,\mathcal{A})$ from Lemma~\ref{lemma:p(n,r,k,A)-bound}, we obtain
\begin{equation}\label{eq:paving_a_large:1}
	\ln p(n,r,a,b)
		\leq \ln \binom{\binom{n}{r-1}}{a} + \frac{a}{r} \left(\ln(n-r+1) + 1 - r + f_{r,0}^{\text{(\ref{lemma:p(n,r,k,A)-bound})}}(n)\right) + \ln \binom{\binom{n}{r}}{b/r}.
\end{equation}
By the lower bound on $a$, we have
\begin{equation}\label{eq:paving_a_large:2}
	\ln \binom{\binom{n}{r-1}}{a}
		=\ln \binom{\binom{n}{r-1}}{\binom{n}{r-1}-a}
		\le \frac{1}{r}\binom{n}{r-1}\frac{r\ln \left(\e\sqrt{\ln(n)}\right)}{\sqrt{\ln(n)}}.
\end{equation}
We split the analysis of $\binom{\binom{n}{r}}{b/r}$ in two cases: small $b$ and large $b$.

If $b \le \frac{1}{\ln^2(n)}\binom{n}{r-1}$, then
\begin{equation}\label{eq:paving_a_large:3}
	\ln \binom{\binom{n}{r}}{b/r}
		\le \frac{1}{r}\binom{n}{r-1} \frac{\ln\left(\e(n-r+1)\ln^2(n)\right)}{\ln^2(n)}
		\le \frac{1}{r}\binom{n}{r-1} \frac{3}{\ln(n)}.
\end{equation}
In this case, upon combining~\eqref{eq:paving_a_large:1}--\eqref{eq:paving_a_large:3} with $a \le \binom{n}{r-1}$, we obtain
\begin{equation*}
	\ln p(n,r,a,b)
		\le \frac{1}{r}\binom{n}{r-1}\left(\ln(n-r+1) + 1 - r + \frac{r\ln\left(\e\sqrt{\ln(n)}\right)}{\sqrt{\ln(n)}} + f_{r,0}^{(\text{\ref{lemma:p(n,r,k,A)-bound}})}(n) + \frac{3}{\ln n}\right),
\end{equation*}
and hence~\eqref{eq:paving_a_large} follows.

On the other hand, if $b \ge \frac{1}{\ln^2(n)}\binom{n}{r-1}$, then
\begin{equation}\label{eq:paving_a_large:5}
	\ln \binom{\binom{n}{r}}{b/r}
		\le \frac{b}{r} \ln \left(\frac{\e\binom{n}{r}}{\frac{1}{r\ln^2(n)}\binom{n}{r-1}}\right)
		\le \frac{b}{r} \ln\left(\e(n-r+1)\ln^2(n)\right)
		\le \frac{\gamma_rb}{r} \left(\ln(n-r+1) + 1 - r\right),
\end{equation}
where the final inequality follows from~\eqref{eq:paving_a_large:n0}.
In this case, it follows upon combining~\eqref{eq:paving_a_large:1}, \eqref{eq:paving_a_large:2} and \eqref{eq:paving_a_large:5} with the bound $a + \gamma_r b \le \binom{n}{r-1}$ from Lemma~\ref{lemma:v0v1} that
\begin{equation*}
	\ln p(n,r,a,b) \le \frac{1}{r}\binom{n}{r-1} \left(\ln(n-r+1)+1-r+\frac{r\ln\left(\e\sqrt{\ln(n)}\right)}{\sqrt{\ln(n)}} + f_{r,0}^{(\text{\ref{lemma:p(n,r,k,A)-bound}})}(n)\right),
\end{equation*}
which implies~\eqref{eq:paving_a_large} and hence concludes the proof.
\end{proof}
We may now complete the proof of Theorem~\ref{thm:p(n,r)}.
\begin{proof}[Proof of Theorem~\ref{thm:p(n,r)}]
We have 
$$p(n,r)= \sum_{a=0}^{\binom{n}{r-1}} \sum_{b=0}^{\binom{n}{r-1}} p(n,r,a,b)\leq \left(1+\binom{n}{r-1}\right)^2 \max_{a,b} p(n,r,a,b).$$
With the bounds obtained as in Lemmas~\ref{lem:paving_a_small} and~\ref{lem:paving_a_large},
$$\max_{a,b}\,\ln p(n,r,a,b)\leq \frac{1}{r}\binom{n}{r-1}\left(\ln(n-r+1)+1-r+f_r^{\text{(\ref{lem:paving_a_large})}}(n)\right)$$
for all  $n\geq \exp\left(\left(\frac{(r-1)(r+1)}{r-3}\right)^2\right)$.
It follows that 
$$\ln p(n,r)\leq 2\ln\left(1+\binom{n}{r-1}\right)~+~\frac{1}{r}\binom{n}{r-1}\left(\ln(n-r+1)+1-r+f_r^{\text{(\ref{lem:paving_a_large})}}(n)\right).$$
As $r\geq 4$, the term $2\ln\left(1+\binom{n}{r-1}\right)\leq 2(r-1)\ln(\e n/(r-1))$ is tiny compared to the upper bound on $\ln p(n,r,a,b)$. Theorem~\ref{thm:p(n,r)} follows.
\end{proof}

\section{Paving matroids of rank 3}\label{sec:p(n,3)}

\subsection{The result}

In this section, we prove the following upper bound on the number $p(n,3)$ of rank-3 paving matroids on a ground set of $n$ elements.

\begin{theorem}\label{thm:pn3-upperbound}
	There exists $\beta < 0$ such that
	\begin{equation*}
		\ln p(n,3) \le \frac{1}{n-2}\binom{n}{3} \ln\left(\e^{1+\beta} n + o(n)\right) \qquad \text{as $n\to\infty$}.
	\end{equation*}
\end{theorem}
Theorem~\ref{thm:pn3-upperbound}, combined with the bound $\beta < -0.65$ which we obtain in Lemma~\ref{lemma:beta-bound} below, prove the upper bound in Theorem~\ref{thm:p(n,3)}. Together with the lower bound on $s(n,3)$ from Proposition~\ref{prop:s(n,r):lower-bound}, this concludes the proof of Theorem~\ref{thm:p(n,3)}. 

We characterise the constant $\beta$ that appears in the upper bound as the value of a calculus-of-variations problem that we now define. Write $h(y) = (1-y)\ln(1-y)$ (and $h(1) = 0$). Let $\Delta \coloneqq \bigl\{(x,y) \in [0,1]^2 : 0 \le y \le \min\{x,1-x\}\bigr\}$, and define the function $F\colon\Delta \to \RR$ by
\begin{equation*}
	F(x,y) = -2 -6xh\left(\frac{y}{x}\right) - 6(1-x)h\left(\frac{y}{1-x}\right)-6y\ln\left(\frac{y}{x(1-x)}\right).
\end{equation*}
Define the functional $\mathcal{F}[u] = \int_0^1 F(x,u(x)) \dd{x}$. We show that
\begin{equation}\label{eq:beta-def}
	\beta = \sup_{u \in \Cspace} \mathcal{F}[u],
\end{equation}
where the supremum is taken over the space $\Cspace$ of all continuously differentiable functions $u$ on $[0,1]$ that satisfy the constraints $\int_0^1 u(x) \dd{x} = 1/6$ and $0\le u(x) \le \min\{x,1-x\}$.

The optimisation problem~\eqref{eq:beta-def} can be solved using standard methods from the calculus of variations.

\begin{lemma}\label{lemma:beta-bound}
	$-0.67 < \beta < -0.65$.
\end{lemma}
\begin{proof}
	Maximising $\mathcal{F}[u]$ subject to the constraint $\int_0^1 u(x) \dd{x} = 1/6$ is a problem of Euler--Lagrange type, and it follows that any extremum must satisfy the Euler--Lagrange equation $\frac{\partial}{\partial u}F(x,u(x)) = \lambda$, where $\lambda$ is a multiplier whose value follows from the constraint. After taking the derivative and rearranging terms, we obtain
	\begin{equation}\label{eq:pn3:euler-lagrange}
		(x-u)(1-x-u) = \lambda'u,
	\end{equation}
	where $\lambda' = (\lambda-6)/6$. Equation~\eqref{eq:pn3:euler-lagrange} is a quadratic equation with solutions
	\begin{equation*}
		u_{\pm}(x; \lambda') = \frac{1}{2}\left(1+\lambda' \pm \sqrt{(1+\lambda')^2 - 4x(1-x)}\right).
	\end{equation*}
	Of the two solutions $u_\pm$, only $u_-$ satisfies the constraint $0 \le u(x) \le \min\{x,1-x\}$. It remains to find $\lambda'$ such that $\int_0^1 u_-(x; \lambda') \dd{x} = 1/6$.
	
	The function $\lambda' \mapsto \int_0^1 u_-(x; \lambda') \dd{x}$ is continuous and decreasing on $[0,\infty)$; moreover, $\int_0^1 u_-(x; 0) \dd{x} = 1/4$, while $\lim_{\lambda' \to\infty} \int_0^1 u_-(x; \lambda') \dd{x} = 0$. It follows that there is a unique $\lambda^*$ for which $\int_0^1 u_-(x; \lambda^*) \dd{x} = 1/6$. Numerical evaluation gives $0.2 < \lambda^* < 0.21$.
	
	The function $\mathcal{F}[u_-(\cdot; \lambda')]$ is strictly increasing in $\lambda'$, from which it follows that
	\begin{equation*}
		-0.67 < \mathcal{F}[u_-(\cdot; 0.2)] < \beta \equiv \mathcal{F}[u_-(\cdot; \lambda^*)] < \mathcal{F}[u_-(\cdot; 0.21)] < -0.65. \qedhere
	\end{equation*}
\end{proof}

\subsection{Good sets}

We obtain Theorem~\ref{thm:pn3-upperbound} as a corollary to a stronger result, which we now describe.
Call a subset $\mathcal{X} \subseteq \binom{[n]}{3}$ \emph{good} if
\begin{enumerate}[(i)]
	\item for any pair of triples $\{a_1 < a_2 < a_3\}$ and $\{b_1 < b_2 < b_3\}$ in $\mathcal{X}$, if $a_2 = b_2$, then $a_1 \neq b_1$ and $a_3 \neq b_3$; and
	\item $|\mathcal{X}| \le \frac{1}{n-2} \binom{n}{3}$.
\end{enumerate}
Let $g(n)$ be the number of good subsets of $\binom{[n]}{3}$.
\begin{theorem}\label{thm:goodsets-upperbound}
	$
		\ln g(n) = \frac{1}{n-2}\binom{n}{3} \ln \left(\e^{1+\beta}n + o(n)\right)
	$ as $n\to\infty$.
\end{theorem}

If $M$ is a paving matroid of rank~$3$ on ground set $E=[n]$, then $\VV(M)$ (as defined in Section~\ref{subsec:p(n,r):encoding}) is good: the first property follows from Lemma~\ref{lemma:p(n,r):encoding-stable} and the fact that $\VV(H)$ consists of consecutive subsets of $H$ for each hyperplane $H$, while the second property is Lemma~\ref{lemma:p(n,r):encoding-small}. As $\VV(M)$ determines $M$, it follows that $p(n,3) \le g(n)$, and hence Theorem~\ref{thm:goodsets-upperbound} implies Theorem~\ref{thm:pn3-upperbound}.

\subsection{Good sets: asymptotics}

In this section, we outline a proof of Theorem~\ref{thm:goodsets-upperbound}, postponing some technical details to the next section.
Let $\rv{\mathcal{X}}$ be a set of $t$ triples in $[n]$, chosen uniformly at random from among all such $t$-sets of triples, and write $\prob[n,t]{}$ for its law. Write $\good$ for the event that $\rv{\mathcal{X}}$ is good. Set $T = \frac{1}{n-2}\binom{n}{3}$, $\mathcal{T} \equiv \mathcal{T}_n = \ZZ \cap [\left(1-4/\ln n\right)T, T]$, and $\overline{\mathcal{T}} = \ZZ \cap [0,\left(1-4/\ln n\right)T)$. Evidently,
\begin{equation*}
	g(n)
		= \sum_{t=0}^T \binom{\binom{n}{3}}{t} \prob[n,t]{\good}.
\end{equation*}
Using the trivial bound $0 \le \prob[n,t]{\good} \le 1$,
\begin{equation*}
	0 \le \sum_{t \in \overline{\mathcal{T}}} \binom{\binom{n}{3}}{t} \prob[n,t]{\good} \le \left(\e^{-3}n+o(n)\right)^T; 
\end{equation*}
thus, in order to prove Theorem~\ref{thm:goodsets-upperbound}, it remains to show that
\begin{equation}\label{eq:p(n,3):toprove}
	\sum_{t \in \mathcal{T}_n} \binom{\binom{n}{3}}{t} \prob[n,t]{\good} = \left(\e^{1+\beta}n+o(n)\right)^T
	\qquad\text{as $n\to\infty$}.
\end{equation}
Most of the technical work to prove~\eqref{eq:p(n,3):toprove} is done in Lemmas~\ref{lemma:p(n,3):prob-approx} and~\ref{lemma:p(n,3):beta-is-max} below, the proofs of which are deferred to the next section.

In what follows, we write $a_{n,t} \logapprox b_{n,t}$ if
\begin{equation*}
	\max_{t \in \mathcal{T}_n} \left|\frac{1}{t} \ln a_{n,t} - \frac{1}{t} \ln b_{n,t}\right| = O\left(\frac{\ln n}{n}\right) \qquad \text{as $n\to\infty$}.
\end{equation*}

Let
\begin{equation*}
	\mathcal{Z}_{n,t} =
		\left\{\vec{z} = (z_2, \ldots, z_{n-1}) \in \ZZ^{n-2} :
			\begin{array}{l}
				0 \le z_i \le \min\{i-1, n-i\} \quad \text{for all $i$}, \\
				\sum_{i=2}^{n-1} z_i = t
			\end{array}\right\}.
\end{equation*}
Recall that $h(y) = (1-y) \ln (1-y)$ for $0 \le y < 1$ and $h(1) = 0$. Define
\begin{equation*}
	f_{n,t}(\vec{z}) = -2 - \frac{1}{t} \sum_{i=2}^{n-1} \biggl[(i-1)h\!\left(\frac{z_i}{i-1}\right) + (n-i)h\!\left(\frac{z_i}{n-i}\right) + z_i \ln\left(\frac{z_i/t}{(i-1)(n-i)/\binom{n}{3}}\right)\biggr].
\end{equation*}

\begin{lemma}\label{lemma:p(n,3):prob-approx}
	$\prob[n,t]{\good} \logapprox \exp\left(t \max\limits_{\vec{z} \in \mathcal{Z}_{n,t}} f_{n,t}(\vec{z})\right)$.
\end{lemma}

\begin{lemma}\label{lemma:p(n,3):beta-is-max}
	$\limsup\limits_{n\to\infty} \max\limits_{t \in \mathcal{T}_n} \max\limits_{\vec{z} \in \mathcal{Z}_{n,t}} f_{n,t}(\vec{z}) = \liminf\limits_{n\to\infty} \max\limits_{T-n+2 \le t \le T} \max\limits_{\vec{z} \in \mathcal{Z}_{n,t}} f_{n,t}(\vec{z}) = \beta$.
\end{lemma}

We are now ready to prove Theorem~\ref{thm:goodsets-upperbound} subject to Lemmas~\ref{lemma:p(n,3):prob-approx} and~\ref{lemma:p(n,3):beta-is-max}.

\begin{proof}[Proof of Theorem~\ref{thm:goodsets-upperbound}]
	As indicated before, it remains to prove~\eqref{eq:p(n,3):toprove}. Let $0 < \varepsilon < |\beta|$ be given. We first turn to proving the upper bound. Let $N \ge \exp(4|\beta|/\varepsilon)$ be so large that $\prob[n,t]{\good} \le \e^{t(\beta + \varepsilon/2)}$ for all $t \in \mathcal{T}_n$ whenever $n \ge N$; such an~$N$ exists by Lemmas~\ref{lemma:p(n,3):prob-approx} and~\ref{lemma:p(n,3):beta-is-max}. For $n \ge N$, we find
	\begin{equation}\label{eq:p(n,3):goodsets:1}
		\sum_{t \in \mathcal{T}_n} \binom{\binom{n}{3}}{t} \le \left(\e(n-2)\right)^T \e^{T(1-4/\ln n)(\beta + \varepsilon/2)} \le \e^{T(1+\beta + \varepsilon + \ln(n-2))}.
	\end{equation}
	We turn to proving the lower bound. For $n = 1, 2, \ldots$, let $\tau_n$ be a maximiser of $t \mapsto \max\limits_{\vec{z} \in \mathcal{Z}_{n,t}} f_{n,t}(\vec{z})$ on $\ZZ\cap[T-n+2, T]$. Let $N'$ be so large that $\prob[n,\tau_n]{\good} \ge \e^{\tau_n(\beta-\varepsilon/3)}$, $(n-2)(1+\ln(n-2))/T \le \varepsilon/3$, and $(T-n+2)^{-1} \ln(T-n+2) < \varepsilon/3$ whenever $n \ge N'$. For $n \ge N'$, we find, using Lemma~\ref{lemma:prelim:binomial-lb},
	\begin{multline}\label{eq:p(n,3):goodsets:2}
		\sum_{t \in \mathcal{T}_n} \binom{\binom{n}{3}}{t} \prob[n,t]{\good}
			\ge \binom{\binom{n}{3}}{\tau_n} \prob[n,\tau_n]{\good}
			\ge \binom{\binom{n}{3}}{\tau_n} \e^{\tau_n(\beta-\varepsilon/3)} \\
			\ge \e^{(T-n+2)(1+\beta-\frac{2}{3}\varepsilon + \ln (n-2))}
			\ge \e^{T(1+\beta-\varepsilon + \ln(n-2))}.
	\end{multline}
	The theorem now follows as~\eqref{eq:p(n,3):goodsets:1}--\eqref{eq:p(n,3):goodsets:2} hold whenever $n \ge \max\{N,N'\}$, and $\varepsilon$ is arbitrarily small.
\end{proof}

\subsection{Good sets: details}

In this section, we prove Lemmas~\ref{lemma:p(n,3):prob-approx} and~\ref{lemma:p(n,3):beta-is-max}, thus finishing the proof of Theorem~\ref{thm:goodsets-upperbound}.

\begin{proof}[Proof of Lemma~\ref{lemma:p(n,3):prob-approx}]
	Recall that $\rv{\mathcal{X}}$ is chosen uniformly at random from the collections of $t$ triples in $[n]$, and that $\good$ denotes the event that $\rv{\mathcal{X}}$ is good. For $i = 2, \ldots, n-1$, let $\rv{Z}_i$ denote the number of triples in $\rv{\mathcal{X}}$ whose middle element is $i$, and write $\rv{\vec{Z}} = (\rv{Z}_2, \ldots, \rv{Z}_{n-1})$. It is easily verified that if $\rv{\mathcal{X}}$ is good, then $\rv{\vec{Z}} \in \mathcal{Z}_{n,t}$.
	
	By conditioning on $\rv{\vec{Z}}$, we obtain
	\begin{equation*}
		\prob[n,t]{\good} = \sum_{\vec{z} \in \mathcal{Z}_{n,t}} \condprob[n,t]{\good}{\rv{\vec{Z}}=\vec{z}} \prob[n,t]{\rv{\vec{Z}}=\vec{z}}.
	\end{equation*}
	As $|\mathcal{Z}_{n,t}| \le t^{n-2}$ and $\frac{1}{t} \log t^{n-2} = O\left(\frac{\log n}{n}\right)$ uniformly in $t \in \mathcal{T}_n$ as $n \to \infty$, it follows that
	\begin{equation}\label{eq:p(n,3):prob-approx:1}
		\prob[n,t]{\good} \logapprox \max_{\vec{z} \in \mathcal{Z}_{n,t}} \condprob[n,t]{\good}{\rv{\vec{Z}}=\vec{z}} \prob[n,t]{\rv{\vec{Z}}=\vec{z}}.
	\end{equation}
	
	We start by analysing the second factor. The random variable $\rv{\vec{Z}}$ has a multivariate hypergeometric distribution, so that (writing $k_i = (i-1)(n-i)$ and $N = \binom{n}{3}$)
	\begin{equation*}
		\prob[n,t]{\rv{\vec{Z}} = \vec{z}} = \binom{N}{t}^{-1} \prod_{i=2}^{n-1} \binom{k_i}{z_i}, \qquad 0 \le z_i \le k_i.
	\end{equation*}
	
	Using Stirling's approximation~\eqref{eq:stirling}, for $\vec{z} \in \mathcal{Z}_{n,t}$,
	\begin{multline*}
		\left|\frac{1}{t} \ln {\prob[n,t]{\rv{\vec{Z}} = \vec{z}}} + \sum_{i:z_i > 0} \frac{z_i}{t} \ln\left(\frac{z_i/t}{k_i/N}\right)\right| \\
		\le \frac{N-t}{t} \left| \sum_{i:z_i > 0} \frac{k_i-z_i}{N-t} \ln\left(\frac{(k_i-z_i)/(N-t)}{k_i/N}\right) \right|\\
		+ \frac{1}{2t}\left|\ln \left(\frac{t(N-t)}{N}\right) + \sum_{i:z_i > 0} \ln\left(\frac{k_i}{z_i(k_i-z_i)}\right)\right|
		+ C \frac{n-1}{t},
	\end{multline*}
	where $C = 3\ln \bigl(\e/\sqrt{2\pi}\bigr)$. In particular, there exists a constant $c > 0$ such that
	\begin{equation}\label{eq:p(n,3):prob-approx:2}
		\left|\frac{1}{t} \ln {\prob[n,t]{\rv{\vec{Z}} = \vec{z}}} + \sum_{i=2}^{n-1} \frac{z_i}{t} \ln\left(\frac{z_i/t}{k_i/N}\right)\right| \le \frac{c \ln n}{n},
	\end{equation}
	for all $n$, and for all $t \in \mathcal{T}_n$ and $\vec{z} \in \mathcal{Z}_{n,t}$.
	
	Finally, we show that
	\begin{equation}\label{eq:p(n,3):prob-approx:3}
		\left|\ln {\condprob[n,t]{\good}{\rv{\vec{Z}}=\vec{z}}} -\sum_{i=2}^{n-1}\left[-2z_i - (i-1)h\left(\frac{z_i}{i-1}\right) - (n-i)h\left(\frac{z_i}{n-i}\right)\right]\right|
			\le 4n\ln(n),
	\end{equation}
	which, together with~\eqref{eq:p(n,3):prob-approx:2} and~\eqref{eq:p(n,3):prob-approx:1} proves the lemma.
	
	Write $\mathcal{G}_i$ for the event that the triples with central element $\le i$ are good. By the chain rule for probabilities,
	\begin{equation*}
		\condprob[n,t]{\good}{\rv{\vec{Z}}=\vec{z}}
			= \condprob[n,t]{\bigcap_{i=2}^{n-1} \mathcal{G}_i}{\rv{\vec{Z}}=\vec{z}}
			= \prod_{i=2}^{n-1} \condprob[n,t]{\mathcal{G}_i}{\bigcap_{j < i} \mathcal{G}_j, \rv{\vec{Z}}=\vec{z}}.
	\end{equation*}
	Fix $2 \le i \le n-1$. Given $\mathcal{G}_j$ for all $j < i$ and $\rv{\vec{Z}}=\vec{z}$, $\mathcal{G}_i$ holds if and only if $\rv{\mathcal{X}}_i = \{\{a_1, a_2, a_3\} \in \rv{\mathcal{X}} : a_2 = i\}$ is good. Each triple in $\rv{\mathcal{X}}_i$ is
	specified by selecting an element that is smaller than $i$ and an element that is larger than $i$, and each of these elements has to be distinct. Thus, there are $(i-1)_{z_i}(n-i)_{z_i}$ ways of selecting the $z_i$ triples with central element $i$, where we use $(x)_k = x(x-1)\dotsm(x-k+1)$ to denote the falling factorial. It follows that
	\begin{equation*}
		\condprob[n,t]{\mathcal{G}_i}{\bigcap_{j < i} \mathcal{G}_j, \rv{\vec{Z}}=\vec{z}}
		= \frac{(i-1)_{z_i} (n-i)_{z_i}}{(i-1)^{z_i}(n-i)^{z_i}}
		= \prod_{k=0}^{z_i-1} \left(1-\frac{k}{i-1}\right)\left(1-\frac{k}{n-i}\right),
	\end{equation*}
	and hence, upon taking logarithms,
	\begin{equation*}
		\ln \condprob[n,t]{\good}{\rv{\vec{Z}}=\vec{z}}
			= \sum_{i=2}^{n-1} \sum_{k=0}^{z_i-1} \left[\ln\left(1-\frac{k}{i-1}\right) + \ln\left(1-\frac{k}{n-i}\right)\right].
	\end{equation*}
	Fix $i$ and $m \in \{i-1, n-i\}$. By concavity of the function $x \mapsto \ln(1-x/m)$,
	\begin{equation*}
		\left|\ln\left(1-\frac{k}{m}\right) - \int_k^{k+1} \ln\left(1-\frac{x}{m}\right) \dd{x}\right| \le \varepsilon_{m,k},
	\end{equation*}
	where
	\begin{equation*}
		\varepsilon_{m,k} =
			\begin{cases}
				\frac{1}{2}\left[\ln\left(1-\frac{k}{m}\right) - \ln\left(1 - \frac{k+1}{m}\right)\right] & \text{if $k = 0, 1, \ldots, m-2$,} \\
				1 & \text{if $k = m-1$}.
			\end{cases}
	\end{equation*}
	Due to the telescoping nature of the $\varepsilon_{m,k}$, upon summing over $k$, we obtain
	\begin{equation*}
		\left|\sum_{k=0}^{z_i-1} \ln\left(1-\frac{k}{m}\right) - z_i \int_0^1\ln\left(1-\frac{xz_i}{m}\right)\right| \le \sum_{k=0}^{m-1} \varepsilon_{m,k} \le 1 + \frac{\ln m}{2} \le 2\ln n,
	\end{equation*}
	Using that $\int_0^1 \ln(1-\alpha x)\dd{x} = -1 - \frac{1-\alpha}{\alpha}\ln(1-\alpha)$, and summing over $m$, this proves~\eqref{eq:p(n,3):prob-approx:3}, and hence completes the proof of the lemma.
\end{proof}

Before proving Lemma~\ref{lemma:p(n,3):beta-is-max}, we require two additional technical results that relate the discrete optimisation problem of Lemma~\ref{lemma:p(n,3):cont-approx} to the continuous optimisation problem~\eqref{eq:beta-def}.

Starting from $\vec{z} \in \mathcal{Z}_{n,t}$, define the step function $z$ associated with $\vec{z}$ by
\begin{equation*}
	z(x) =
	\begin{cases}
		0 & \text{if $x \le \frac{1}{n}$ or $x > 1-\frac{1}{n}$,} \\
		\frac{z_i}{n} & \text{if $\frac{i-1}{n} < x \le \frac{i}{n}$.}
	\end{cases}
\end{equation*}

Writing $i_n(x) = \lceil xn \rceil$, it follows that $z(x) = z_{i_n(x)}/n$ (whenever $z_{i_n(x)}$ exists).

\begin{lemma}\label{lemma:p(n,3):step-function}
	For all $\varepsilon > 0$, there exists $N^{\text{(\ref{lemma:p(n,3):step-function})}} \equiv N^{\text{(\ref{lemma:p(n,3):step-function})}}(\varepsilon)$ such that for all $n \ge N^{\text{(\ref{lemma:p(n,3):step-function})}}$, $t \in \mathcal{T}_n$, and $\vec{z} \in \mathcal{Z}_{n,t}$, if $z$ is the step function associated with $\vec{z}$, then $|f_{n,t}(\vec{z}) - \mathcal{F}[z]| < \varepsilon$. 
\end{lemma}

\begin{proof}
	Replacing the sum by an integral, we have
	\begin{multline*}
		f_{n,t}(\vec{z}) =\\ -2 - \frac{n^2}{t} \int\limits_0^1 \biggl(\frac{i_n(x) - 1}{n} g\left(\frac{z(x) n}{i_n(x) - 1}\right) + \frac{n - i_n(x)}{n} g\left(\frac{z(x) n}{n-i_n(x)}\right)
			+ z(x) \ln\frac{z(x) n(n-2)}{(i_n(x)-1)(n-i_n(x))} \biggl)\dd{x}.
	\end{multline*}
	By continuity of the integrand, it follows that,
	for all $\vec{z} \in \mathcal{Z}_{n,t}$,
	\begin{equation*}
		\left|f_{n,t}(\vec{z}) - \mathcal{F}[z]\right| < \varepsilon
	\end{equation*}
	provided that $n$ is sufficiently large.
\end{proof}

The next lemma shows that $f_{n,t}(\vec{z})$ can be approximated to arbitrary precision by the functional $\mathcal{F}$. Recall that $\Cspace$ is the space of all continuously differentiable functions $u\colon[0,1]\to\mathbb{R}$ that satisfy the constraints $\int_0^1 u(x) \dd{x} = 1/6$ and $0 \le u(x) \le \min\{x,1-x\}$.

\begin{lemma}\label{lemma:p(n,3):cont-approx}
	For all $\varepsilon > 0$ there exists $N^{\text{(\ref{lemma:p(n,3):cont-approx})}} \equiv N^{\text{(\ref{lemma:p(n,3):cont-approx})}}(\varepsilon)$ such that for all $n \ge N^{\text{(\ref{lemma:p(n,3):cont-approx})}}$, all $t \in \mathcal{T}_n$, and all $\vec{z} \in \mathcal{Z}_{n,t}$ there exists $\ztail \in \Cspace$ such that $|f_{n,t}(\vec{z}) - \mathcal{F}[\ztail]| < \varepsilon$.
\end{lemma}

\begin{proof}
	We construct $\ztail$ in three steps. In the first step, we construct an approximation of $\vec{z}$ by a step function $z$. In the second step, we tweak $z$ so that its integral evaluates to $1/6$ which yields another function $\hat{z}$. In the third step, we smooth $\hat{z}$ using convolution to obtain $\ztail$.
	
	\noindent\uline{Step 1.} Let $z$ be the step function associated with $\vec{z}$. By Lemma~\ref{lemma:p(n,3):step-function}, we can ensure that
	\begin{equation}\label{eq:pn3:approx-1}
		\text{for all $n \ge N^{(\ref{lemma:p(n,3):step-function})}(\varepsilon/3)$:} \qquad \left|f_{n,t}(\vec{z}) - \mathcal{F}[z]\right| < \varepsilon/3.
	\end{equation}
	
	\noindent\uline{Step 2.} Note that $I_1 \coloneqq \int_0^1 z(x) \dd{x} = t/n^2 < 1/6$. Let $I_2 \coloneqq 1/2 - 1/n^2$, and let $\lambda$ be such that $(1-\lambda)I_1 + \lambda I_2 = 1/6$. For large $n$, $0 \le \lambda < 5/\ln n \le 1$. Define
	\begin{equation*}
		\hat{z}(x) =
		\begin{cases}
			0 & \text{if $x \le \frac{1}{n}$ or $x > 1 - \frac{1}{n}$,} \\
			(1-\lambda)z(x) + \lambda\min\{x,1-x\} & \text{otherwise}.
		\end{cases}
	\end{equation*}
	
	By construction, $\int_0^1 \hat{z}(x) \dd{x} = 1/6$, while
	\begin{equation*}
		0 \le z(x) \le \hat{z}(x) \le \min\{x,1-x\} \qquad \text{for all $x \in [0,1]$},
	\end{equation*}
	and the pointwise difference between $z$ and $\hat{z}$ satisfies
	\begin{equation*}
		|z(x) - \hat{z}(x)| \le \lambda < \frac{5}{\ln n}\qquad \text{for all $x \in [0,1]$}.
	\end{equation*}
	Hence, by uniform continuity, there exists $N^{\text{(\ref{lemma:p(n,3):cont-approx})}}_1(\varepsilon)$ such that
	\begin{equation}\label{eq:pn3:approx-2}
		\text{for all $n \ge N^{\text{(\ref{lemma:p(n,3):cont-approx})}}_1(\varepsilon)$:} \qquad |\mathcal{F}[z] - \mathcal{F}[\hat{z}]| < \varepsilon/3.
	\end{equation}
	
	\noindent\uline{Step 3.} Define
	\begin{equation*}
		K_\delta(y) =
			\begin{cases}
				\frac{1}{\eta} \exp\left(\frac{1}{y^2-\delta^2}\right) & \text{if $|y| < \delta$,} \\
				0 & \text{otherwise,}
			\end{cases}
	\end{equation*}
	where $\eta = \eta(\delta) = \int_{-\delta}^{\delta} \exp\left(\frac{1}{y^2-\delta^2}\right)\dd{y}$.
	Note that $K_\delta$ is smooth, nonnegative, and has support $(-\delta,\delta)$.
	Define $\ztail = \hat{z} * K_{1/n^2}$, i.e.
	\begin{equation*}
		\ztail(x) = \int\limits_{-\infty}^{\infty} \hat{z}(x-y) K_{1/n^2}(y) \dd{y}, \qquad x \in [0,1],
	\end{equation*}
	where, for convenience, we use $\hat{z}(x) = 0$ whenever $x < 0$ or $x > 1$.
	The following properties of $\ztail$ follow from elementary properties of convolutions:
	\begin{enumerate}[(a)]
		\item $\ztail$ is smooth on $[0,1]$, and
		\item $\int_0^1 \ztail(x) \dd{x} = \int_0^1 \hat{z}(x) \dd{x} = 1/6$.
	\end{enumerate}
	Moreover, since $0 \le \hat{z}(x) \le \min\{x,1-x\}$ for all $x \in [0,1]$ and $K_\delta(y)$ is symmetric about $y=0$,
	\begin{enumerate}[(a)]
		\setcounter{enumi}{2}
		\item $0 \le \ztail(x) \le \min\{x,1-x\}$ for all $x \in [0,1]$.
	\end{enumerate}
	Thus, $\ztail \in \Cspace$.
	
	By construction, $\ztail(x) = \hat{z}(x)$ for all $x$ except for a set of (Lebesgue) measure at most $c_2/n^2$. It follows that there exists $N^{\text{(\ref{lemma:p(n,3):cont-approx})}}_2(\varepsilon)$ such that
	\begin{equation}\label{eq:pn3:approx-3}
		\text{for all $n \ge N^{\text{(\ref{lemma:p(n,3):cont-approx})}}_2(\varepsilon)$:} \qquad |\mathcal{F}[\hat{z}] - \mathcal{F}[\ztail]| < \varepsilon/3.
	\end{equation}
	
	The lemma holds with $N^{\text{(\ref{lemma:p(n,3):cont-approx})}}(\varepsilon) \coloneqq \max\left\{N^{\text{(\ref{lemma:p(n,3):step-function})}}(\varepsilon/3), N^{\text{(\ref{lemma:p(n,3):cont-approx})}}_1(\varepsilon), N^{\text{(\ref{lemma:p(n,3):cont-approx})}}_2(\varepsilon)\right\}$, as \eqref{eq:pn3:approx-1}--\eqref{eq:pn3:approx-3} imply that $|f_{n,t}(\vec{z}) - \mathcal{F}[\ztail]| < \varepsilon$ whenever $n \ge N^{\text{(\ref{lemma:p(n,3):cont-approx})}}(\varepsilon)$.
\end{proof}

We are now ready to prove Lemma~\ref{lemma:p(n,3):beta-is-max}.

\begin{proof}[Proof of Lemma~\ref{lemma:p(n,3):beta-is-max}]
	Let $\varepsilon > 0$ be given. By Lemma~\ref{lemma:p(n,3):cont-approx}, if $n \ge N^{\text{(\ref{lemma:p(n,3):cont-approx})}}(\varepsilon)$, then for all $t \in \mathcal{T}_n$ and $\vec{z} \in \mathcal{Z}_{n,t}$, there exists $\ztail \in \Cspace$ such that
	\begin{equation*}
		f_{n,t}(\vec{z}) \le \mathcal{F}[\ztail] + \varepsilon \le \beta + \varepsilon.
	\end{equation*}
	As the right-hand side does not depend on $n$, $t$, or $\vec{z}$, this proves the upper bound in the lemma.
	
	We now turn to proving the corresponding lower bound.
	Let $\ztail$ be such that ${\mathcal{F}[\ztail] > \beta - \varepsilon/3}$. For given $n \ge 3$, define the sequence $\vec{z} = (z_2, \ldots, z_n)$ as
	\begin{equation*}
		z_i =
		\begin{cases}
			\left\lfloor 6T \int_0^{2/n} \ztail(x) \dd{x}\right\rfloor & \text{if $i=2$,} \\[1ex]
			\left\lfloor 6T \int_{1-2/n}^1 \ztail(x) \dd{x}\right\rfloor & \text{if $i=n-1$,} \\[1ex]
			\left\lfloor 6T \int_{(i-1)/n}^{i/n} \ztail(x) \dd{x}\right\rfloor & \text{otherwise},
		\end{cases}
	\end{equation*}
	and set $t = \sum_{i=2}^{n-1} z_i$. It is easily verified that $T - n + 2 \le t \le T$ and that $\vec{z} \in \mathcal{Z}_{n,t}$. Let $z$ be the step function associated with $\vec{z}$. By Lemma~\ref{lemma:p(n,3):step-function}, $|f_{n,t}(\vec{z}) - \mathcal{F}[z]| < \varepsilon/3$ whenever $n \ge N^{\text{(\ref{lemma:p(n,3):step-function})}}(\varepsilon/3)$. Since $\ztail$ is continuously differentiable on a compact set, it has bounded derivative; using a Taylor expansion of $\ztail$ around $x$, we find that there is a constant $c > 0$ such that $|z(x) - \ztail(x)| \le c/n$ for all $x \in [0,1]$. By continuity, there exists $N^{\text{\ref{lemma:p(n,3):beta-is-max}}}(\varepsilon)$ such that $|\mathcal{F}[z] - \mathcal{F}[\ztail]| < \varepsilon/3$ for all $n \ge N^{\text{\ref{lemma:p(n,3):beta-is-max}}}(\varepsilon)$. Combining the three estimates, we find that
	\begin{equation*}
		\bigl|f_{n,t}(\vec{z}) - \beta\bigr|
			\le \bigl|f_{n,t}(\vec{z}) - \mathcal{F}[z]\bigr|
				+ \bigl|\mathcal{F}[z] - \mathcal{F}[\ztail]\bigr|
				+ \bigl|\mathcal{F}[\ztail] - \beta\bigr| < \varepsilon.
	\end{equation*}
	It follows that for $n \ge \max\{N^{\text{(\ref{lemma:p(n,3):step-function})}}(\varepsilon/3), N^{\text{\ref{lemma:p(n,3):beta-is-max}}}(\varepsilon)\}$, there exist $t$, $T-n+2 \le t \le T$ and $\vec{z} \in \mathcal{Z}_{n,t}$ such that $f_{n,t}(\vec{z}) \ge \beta - \varepsilon$; this proves the lower bound.
\end{proof}

\section{Final remarks}\label{sec:finalremarks}

We have established tight bounds on the number of (paving) matroids of rank at least~4 and sharpened the best known upper bound for (paving) matroids of rank~3.
In particular, Theorem~\ref{thm:p(n,r)} shows that
\begin{equation*}
	\ln p(n,r) = \frac{1}{n-r+1}\binom{n}{r}\left(\ln(n-r+1) + 1 - r + o(1)\right)\text{ as $n\to\infty$},
\end{equation*}
for all fixed $r \ge 4$. Up to the $o(1)$-term, the same result holds for $s(n,r)$ and $m(n,r)$; so for matroids of rank at least~4, asymptotic enumeration is settled at this level of precision. 

For rank $r=3$, a larger gap remains. While for sparse paving matroids an estimate of similar precision as for higher rank holds, for paving matroids the ``error term'' is $\Theta(1)$ rather than $o(1)$. More precisely, Theorem~\ref{thm:s(n,r)} and Theorem~\ref{thm:p(n,3)} show that, when $r=3$,
\begin{multline*}
	\frac{1}{n-r+1}\binom{n}{r}\left(\ln(n-r+1)+1-r+o(1)\right)
	= \ln s(n,r) \\
	\le \ln p(n,r)
	\le\frac{1}{n-r+1}\binom{n}{r}\left(\ln(n-r+1)+c\right)
\end{multline*}
as $n\to\infty$, where $c = 0.35 > -2 = 1-r$ for sufficiently large $n$. We are not entirely convinced that the constant~$c$ in our upper bound is best possible, but we do believe that in rank~3 the gap between $p(n,r)$ and $s(n,r)$ is more pronounced than in higher rank. Specifically, let $p_k(n,r)$ denote the number of paving matroids without hyperplanes of cardinality $>r+k$. The techniques from Section 4 show that $p(n,r)\approx p_0(n,r)=s(n,r)$ if  $r\ge 4$, but not if $r=3$. We conjecture that this reflects the following underlying truth. 
\begin{conjecture}\label{conj:pn3} Let $r=3$. There is a constant $c>-2$ such that
\begin{equation*}
	\ln p(n,r)\approx \ln p_1(n,r) = \frac{1}{n-r+1} \binom{n}{r} \left(\ln(n-r+1) + c + o(1)\right)\text{ as }n\to\infty.
\end{equation*}
\end{conjecture}

\begin{note}
	It was recently shown by Kwan, Sah, and Sawhney~\cite{KwanSahSawhney2021} that Conjecture~\ref{conj:pn3} holds with $c = \frac{\sqrt{3}}{2}-3 + \ln\left(\frac{1+\sqrt{3}}{2}\right) \approx -1.82207$.
\end{note}

\appendix

\section*{Acknowledgements}

We would like to thank Erlang Surya and Lutz Warnke for pointing out a mistake in an earlier version of Lemma~\ref{lemma:p(n,r,k,A)-bound}, as well as the two anonymous referees for their careful reading of the manuscript.

\bibliographystyle{alpha}
\bibliography{bib}

\begin{aicauthors}
\begin{authorinfo}[rvdh]
	Remco van der Hofstad\\
	Department of Mathematics and Computer Science\\
	Eindhoven University of Technology\\
	Eindhoven, The Netherlands
\end{authorinfo}
\begin{authorinfo}[rp]
	Rudi Pendavingh\\
	Department of Mathematics and Computer Science\\
	Eindhoven University of Technology\\
	Eindhoven, The Netherlands
\end{authorinfo}
\begin{authorinfo}[jvdp]
	Jorn van der Pol\\
	Department of Combinatorics and Optimization\\
	University of Waterloo\\
	Waterloo, Ontario, Canada
\end{authorinfo}
\end{aicauthors}
\end{document}